\documentclass[11pt,reqno]{amsart}
\usepackage[latin1]{inputenc}
\usepackage{epsfig}
\usepackage{mathrsfs}
\usepackage{color}
\usepackage[british,english]{babel}
\usepackage{amsthm}
\usepackage{amsmath}
\usepackage{amsfonts}
\usepackage{amssymb}
\usepackage{graphicx}
\newtheorem{corollary}{Corollary}[section]

\newtheorem{lemma}[corollary]{Lemma}
\newtheorem{proposition}[corollary]{Proposition}
\newtheorem{remark}[corollary]{Remark}
\newtheorem{theorem}[corollary]{Theorem}
\newfont{\sBlackboard}{msbm10 scaled 900}

\newcommand{\mylabel}[1]{\label{#1}
            \ifx\undefined\stillediting
            \else \fbox{$#1$}\fi }
\newcommand{\BE}{\begin{equation}}

\newcommand{\EEQ}{\end{equation}}
\newcommand{\rfb}[1]{\mbox{\rm
   (\ref{#1})}\ifx\undefined\stillediting\else:\fbox{$#1$}\fi}

\newfont{\Blackboard}{msbm10 scaled 1200}

\newfont{\roma}{cmr10 scaled 1200}

\def\CC{\rm \hbox{C\kern-.56em\raise.4ex
         \hbox{$\scriptscriptstyle |$}\kern+0.5 em }}

\newcommand{\mm}    {{\hbox{\hskip 0.5pt}}}

\newcommand{\bluff} {{\hbox{\raise 15pt \hbox{\mm}}}}
scaled\magstep2

\makeatletter
\def\section{\@startsection {section}{1}{\z@}{-3.5ex plus -1ex minus
    -.2ex}{2.3ex plus .2ex}{\large\bf}}
\makeatother
\def\be{\begin{equation}}
\def\ee{\end{equation}}

\begin{document}
\thispagestyle{empty}
\title{
Carleman Estimates and
% applications to
null controllability of coupled degenerate systems}
\author{E. M. Ait Ben Hassi, \,F. Ammar Khodja, \, A. Hajjaj, \, L. Maniar}
% \thanks{D\'epartement de Math\'ematiques et Informatique,
%Facult\'e des Sciences et techniques , Universit\'e Hassan 1er ,
%Settat 26000,
%e-mail~: hajjaj$\_{} $fsts@yahoo.fr}
%\author{L. Maniar }
%\thanks{D\'epartement de Math\'ematiques, Facult\'e
%des Sciences Semlalia, Universit\'e Cadi Ayyad, Marrakech 40000,
%B.P. 2390, Maroc, e-mail : m.benhassi@ucam.ac.ma, \,
% maniar@ucam.ac.ma}
\date{\today}
% \thanks{We thank F. Alabou-Boussouira and P. Cannarsa that encouraged us and appreciated the idea of choosing weight function independent
% of the coefficient diffusion.}
\begin{abstract} In this paper, we  study  the null
controllability  of  weakly degenerate
 coupled parabolic systems with two different diffusion coefficients and  one control force. To obtain this aim, we develop first new  global Carleman estimates for degenerate parabolic equations with  weight functions different from the ones of \cite{bouss},
  \cite{CanMarVan2008}
  % ,  \cite{Cannarsa3}
  and \cite{Martinez1}.
 % for one parabolic equation,
  % with weight functions independent of the diffusion coefficient.
 % We study also coupled systems with one control force.
\end{abstract}
\maketitle
\section{Introduction}
 This paper is concerned with the null
controllability for the coupled degenerate
parabolic systems
% with two different diffusion coefficients
\begin{align}
&u_{t}
-(x^{\alpha_1}u_x)_x + b_{11}(t,x) u+b_{12}(t,x) v=h(t,x) 1_{\omega} ,&  (t,x)\in
\times(0,1),\label{eq:1}
\\
&v_{t} -(x^{\alpha_2}v_x)_x +b_{22}(t,x) v+b_{21}(t,x) u=0,&  (t,x)\in
\times(0,1),
\\
 &u(t, 0) = u(t, 1) = \,\,
v(t, 0) = v(t, 1) = 0, & t\in
(0,T),\\
  &u(0, x) =u_0(x),\,\,
v(0, x) = v_0( x) ,& x\in (0,1), \label{eq:2}
\end{align}
where $\omega =(a,b)$ is an open subset of $(0,1)$,  $h
\in L^2((0, T)\times[0, 1))$, $(u_0,\, v_0) \in L^2 (0,1)\times L^2 (0,1) $,
$(\alpha_1,\alpha_2)\in(0,1)^2$
and $b_{ij}\in L^\infty((0,T)\times(0,1)), \,\,  i,j=1,2$.

Controllability properties of nondegenerate parabolic equations have been widely studied, see
\cite{Cab_Men_Zua}, \cite{DeT_Zua}, \cite{FaPuZu}, \cite{FaRu1}, \cite{FaRu2}, \cite{Cara}, \cite{Cara_Zua}, \cite{Leb_Rob}, \cite{Lo_Zh_Zu}, \cite{Mi_Zu}, \cite{Rus}, \cite{Tat}, using several  techniques in particular the Carleman estimates. In \cite{bouss}, \cite{CanMarVan2008}, \cite{Martinez1} new Carleman estimates were developed for degenerate parabolic equations and used  to show observability inequalities of the adjoint degenerate problems and then obtain the null controllability. Recently, in \cite{CaToYa} Cannarsa et al. established a local Carleman estimate and deduced unique continuation and boundary approximate controllability for weakly degenerate equations.

The null controllability of coupled parabolic systems  was  studied  for example in
\cite{Khoudja1}, \cite{Khoudja2}, \cite{De0}, \cite{Burgos1}, \cite{Burgos2}, \cite{Gu} in the nondegenerate case.
In \cite{Liu}, Liu et al. considered parabolic cascade systems, $b_{12}=0$, with degeneracy in only
one equation, using the nondegenerate Carleman estimate of Fursikov and Imanuvilov
\cite{Fursikov} and an approximation argument as in \cite{CaMaVa}. In
\cite{Ca-Te}, Cannarsa and De Teresa studied the null controllability of cascade degenerate linear systems with the same diffusion coefficient,
i.e.,   $\alpha_1=\alpha_2$, and with the particular coupling term $b_{21}=1_O$ for some open set $O\subset (0,1)$. 
In \cite{bahm}, we studied the null controllability for degenerate cascade systems with general coupling terms and   two different diffusion coefficients. We used a  Carleman estimate from \cite{bouss}, and chose carefully appropriate parameters in the weight functions
$\varphi_1(t,x)= \frac{\lambda_1(x^{2-{\alpha_1}}-d_1)}{t^4(T-t)^4}$ and $\varphi_2(t,x)= \frac{\lambda_2(x^{2-{\alpha_2}}-d_2)}{t^4(T-t)^4}$ to obtain the inequality $e^{s\varphi_1}\leq C e^{s\varphi_2}$ to absorb the coupling term.

For general degenerate systems \eqref{eq:1}-\eqref{eq:2},  we need the uniform equivalence $e^{s\varphi_1} \equiv e^{s\varphi_2}$.  But this  occurs if and only if $\alpha_1=\alpha_2$. To overcome this problem we propose in this paper a common  weight function
$\varphi(t,x)= \frac{\lambda(x^{2-{\beta}}-d)}{t^k(T-t)^k}$ for some $\beta$ in terms of $\alpha_1$ and $\alpha_2$.
Then, the first step in this paper is  to  show new Carleman estimates for the following degenerate parabolic equation
 \begin{align}
 & y_{t} -(x^{\alpha}y_x)_x = f(t,x), \,\,\,\,\,
 & (t,x)\in (0,T)\times(0,1), \label{1eq1} \\
& y(t,0)=y(t,1)=0, \,\,
& t\in(0,T), \\
& y(0,x)=y_0,
&  x\in (0,1), \label{1eq2}
 \end{align}
 with the weight function $\varphi(t,x)= \frac{\lambda(x^{2-{\beta}}-d)}{t^k(T-t)^k}$ with
 $d,\lambda$ and $k$ constants to be specified later. To prove our Carleman estimates,  we need to show the following fundamental Hardy-Poincar\'e inequality
\begin{align}\label{Hardy1}
 \int_0^1x^{\gamma-2}v^2dx\leq C_\gamma \int_0^1x^\gamma v_x^2dx \quad\quad \mbox{where}\,\,\,\,
C_\gamma=\frac{4}{(1-\gamma)^2}
\end{align}
for $\gamma <1$, and $v$ satisfying $v(0)=0$ and $\int_0^1x^\gamma v_x^2dx <+\infty$. This result was proved in \cite{bouss}, \cite{CanMarVan2008} and \cite{Martinez1} for $0<\gamma<2, \gamma\neq 1$. But, for our Carleman estimates we need this inequality for negative $\gamma$, see Lemma  \ref{Hardy2}.
 This will allow us to deduce Carleman estimates for the adjoint coupled degenerate system
\begin{align}
& U_{t} -(x^{\alpha_1}U_x)_x + b_{11}(t,x) U+b_{21}(t,x) V=0, &  (t,x)\in (0,T)\times(0,1), \label{adjfirst}\\
& V_{t}-(x^{\alpha_2}V_x)_x +b_{22}(t,x) V+ b_{12}(t,x) U=0, & (t,x)\in (0,T)\times(0,1),\label{21}\\
& U(t,1)=U(t,0)=\, V(t,1)=V(t,0)=0 ,& t\in (0,T),\label{23}\\
& U(0,x)=U_0(x), V(0,x)=V_0(x), & x\in (0,1)\label{adjend},
\end{align}
and then its observability inequality. Using a standard argument, we obtain the null controllability of \eqref{eq:1}-\eqref{eq:2}.  By a linearization argument and fixed point, see for example \cite{bahm}, \cite{bouss}, \cite{Can-Gen-2006}, \cite{Zuazua99} one can show easily the null controllability of semilinear  degenerate coupled systems.

 This paper is organized as follows. Section 2 is devoted to the well-posedness of the coupled degenerate systems. In section 3, we establish our new Carleman estimates for  degenerate parabolic equations and deduce similar estimates for the coupled degenerate  systems.
 In section 4, we deduce observability inequality and null controllability results. In appendix, we give summarized proofs of Caccioppoli and Hardy-Poincar\'e  inequalities.

\section{Well-posedness}
In order to study the well-posedness of the system
\eqref{eq:1}-\eqref{eq:2}, we introduce the weighted spaces
 $$
 H_{\alpha_i}^1(0, 1)
:=\Big\{ u \in L^2(0, 1) :\, u \,\mbox{ is abs. continuous in}\,
[0, 1],\,\, x^{\alpha_i/2} u_x \in L^2(0, 1)\, \mbox{and}\, u(0)=u(1) = 0\Big \}
$$
with the norm
$\|u\|^2_{ H^1_{\alpha_i}(0, 1)} := \|u\|^2_{L^2(0,1)} +
\|x^{\alpha_i/2}u_x\|^2_{ L^2(0,1)}$  and
 $$H^2_{\alpha_i} (0, 1):=\Big\{ u \in H^1_{\alpha_i}(0, 1) \,: \, x^{\alpha_i}u_x
\in H^1(0, 1)\Big\}$$
 with the norm $$\|u\|^2_{H^2_{\alpha_i}(0, 1)} := \|u\|^2 _{H^1_{\alpha_i}(0, 1)}
+ \|(x^{\alpha_i}u_x)_x\|^2_{ L^2(0,1)}.$$
We define the operator $(A_i,D(A_i))$ by
$$
A_iu := (x^{\alpha_i}u_x)_x ,\, \, \,  u \in D(A_i) = H^2_{\alpha_i}(0, 1), \; i=1,2.
$$
 We recall the following properties of $(A_i,D(A_i))$.
\begin{proposition} (\cite {cmp}, \cite{CaMaVa}).  For $i=1,2$, the operator $A_i : D(A_i) \longrightarrow L^2(0, 1)$  is closed, 
self-adjoint, negative and   with dense domain.
\end{proposition}
 In the Hilbert space  $\mathbb{H}:= L^2(0,1) \times L^2(0, 1) $,
 the  system \eqref{eq:1}-\eqref{eq:2} can be
transformed in the following Cauchy problem
$$
(CP)\quad
\begin{cases} X'(t) =\mathcal{A}X(t) +B(t)X(t)+G(t),\\
X(0)=\left(\begin{smallmatrix}
u_0\\
v_0
\end{smallmatrix}\right),
\end{cases}
$$ where $X(t)=\left(\begin{smallmatrix}
u(t)\\
v(t)
\end{smallmatrix}
\right),$ $\mathcal{A}=\left(\begin{matrix}
A_1&0\\
0&A_2
\end{matrix}
\right) $, $D(\mathcal{A})=D(A_1)\times D(A_2), \,\, G(t)=\left(\begin{smallmatrix}
h(t,x) 1_\omega\\
0 \end{smallmatrix}
\right)$, 
and 
$$
B(t)=\left(\begin{matrix}
M_{b_{11}(t)}&M_{b_{12}(t)}\\
M_{b_{21}(t)}&M_{b_{22}(t)}
\end{matrix}
\right), \,\, \mbox{where}\,\, M_{b_{ij}(t)}u=b_{ij}(t)u.
$$
As the operator  $\mathcal{A}$
 is diagonal  and since $B(t)$ is a bounded perturbation, the following wellposedness and regularity
results hold.
\begin{proposition}\label{estimsemigroup}
(i) \, The operator $\mathcal{A}$ generates a contraction strongly
continuous semigroup $(T(t))_{t\geq0}.$
\\
(ii)\,  For all $h\in L^2((0,T)\times (0,1))$ and $(u_0,v_0)\in L^2(0,1)\times L^2(0,1)$ there exists a unique mild solution $(u,v)\in X_T:=C\left( [0,T],L^2(0,1)\times L^2(0,1) \right)
\cap L^2\left(0,T;H^1_{\alpha_1}\times H^1_{\alpha_2}\right)$ of \eqref{eq:1}-\eqref{eq:2} satisfying
\begin{align}\label{estiminitialdata}
\sup_{[0,T]}\|(u,v)(t)\|^2_{L^2\times L^2} + \int_0^T  \| (x^{\frac{\alpha_1}{2}}u_x ,x^{\frac{\alpha_2}{2}}v_x) \|^2_{L^2} \, dt
\nonumber \\
\leq C_T\left(\|(u_0,v_0)\|^2_{L^2\times L^2}+
\|h\|^2_{L^2((0,T)\times (0,1))}\right)
\end{align}
for a constant $C_T>0$.
Morover, if $(u_0,v_0)\in H^1_{\alpha_1}\times H^1_{\alpha_2}$ then,  $(u,v)\in Y_T:=C\left([0,T],H^1_{\alpha_1}\times H^1_{\alpha_2}\right)
\cap H^1\left(0,T;L^2(0,1)\times L^2(0,1)\right)\cap L^2\left(0,T;H^2_{\alpha_1}\times H^2_{\alpha_2}\right )$ and
$$
\sup_{[0,T]}\|(u,v)(t)\|^2_{H^1_{\alpha_1}\times H^1_{\alpha_2}} + \int_0^T \left( \| (u_t,v_t) \|^2_{L^2} + \| ((x^{\alpha_1}u_x)_x,(x^{\alpha_2}v_x)_x) \|^2_{L^2} \right) \, dt $$
$$
\leq C_T\left(\|(u_0,v_0)\|^2_{H^1_{\alpha_1}\times H^1_{\alpha_2}}+
\|h\|^2_{L^2((0,T)\times (0,1))}\right)
$$
for a constant $C_T>0$.
\end{proposition}
%
%%%%%%%%%%%%%%%%%%%%%%%%%%%%%%%%%%%%%%%%%%%%%%%%%%%%%%%%%%%%%%%%%%%%%%%%%%%%%%%%%%%%
%
                        \section{Carleman estimates}
%
%%%%%%%%%%%%%%%%%%%%%%%%%%%%%%%%%%%%%%%%%%%%%%%%%%%%%%%%%%%%%%%%%%%%%%%%%%%%%%%%%%%
%
In this section we prove new Carleman estimates for the adjoint system \eqref{adjfirst}-\eqref{adjend}.
For this, let
 $\omega':=(a',b')\Subset \omega$
 and let
us introduce the weight functions :
$\varphi(t,x):=\Theta(t)\psi(x)$;
\,\, $\Theta(t):=\dfrac{1}{t^k(T-t)^k};$ \,\, $\psi(x):=\lambda\left(x^{2-\beta}-d\right)$; \quad
 $\Phi(t,x)=\Psi(x)\Theta(t)$; \quad $\Psi (x):=\left(e^{\rho \sigma(x)}-e^{2\rho||\sigma||_\infty}\right)$; \quad
 $\phi(t,x)=e^{\rho\sigma(x)}\Theta(t)$;
where  $\sigma$ is a function in $C^2([a',1])$
  satisfying $\sigma(x)>0$ in $(a',1)$, $\sigma(a')=\sigma(1)=0$ and $\sigma_x(x)\neq 0$ in
  $[a',1]\backslash\omega_0$  for some open $\omega_0\Subset (a',1)$ and the parameters $d$, $\rho$, $\lambda$ and $k$ are chosen such that $d\geq 5$; $\rho > \frac{4ln 2}{||\sigma||_\infty}$, $\frac{e^{2\rho||\sigma||_\infty}}{d-1} < \lambda < \frac{4}{3d}(e^{2\rho||\sigma||_\infty}-e^{\rho||\sigma||_\infty})$ and $k\geq 4$.
  \begin{remark}\label{rk1}
  \begin{itemize}
  \item
  These weight functions are independent of the diffusion coefficient. This play a crucial role to study coupled system of non cascade form.
  \item
  The existence of the function $\sigma$ was proved for example in \cite{Fursikov} using Morse functions. But in 1-dimension one can show this easily using cut-off functions.
  \item
  If $d\geq 5$ and $\rho > \frac{4ln 2}{||\sigma||_\infty}$ then the interval $\left]\frac{e^{2\rho||\sigma||_\infty}}{d-1}\, , \frac{4}{3d}(e^{2\rho||\sigma||_\infty}-e^{\rho||\sigma||_\infty})\right[$ is not empty. We can then choose $\lambda$ in this interval.
   \item
   For this choice of the parameters $d,\, \rho$ and $\lambda $ the weight functions $\varphi$ and $\Phi$
   satisfy the following inequalities which are needed in the sequel
   \begin{align}\label{comparaison phi}
 \frac43 \Phi < \varphi < \Phi  \,\, \mbox{on } (0,T)\times(0,1).
 \end{align}
 \item
 For nondegenerate problems one needs the following estimates see e.g. \cite{Fursikov}
 \begin{align}\label{majorTheta1}
 \lim_{t\rightarrow O^+} \Theta(t)=\lim_{t\rightarrow T^-} \Theta(t)= +\infty,\quad
 \Theta(t)\geq c_1, \quad |\dot{\Theta}|\leq c_2 \Theta^2, \quad |\ddot{\Theta}|\leq c_3 \Theta^3.
 \end{align}
 and this is satisfied for all $k\geq 1$ with $c_1=(2/T)^{2k}$, $c_2=kT(T/2)^{2(k-1)}$, $c_3=k(k+1)T^2(T/2)^{4(k-1)}$.
 \item
  For the degenerate case one needs in addition the estimate
  \begin{align}\label{majorTheta2}
 |\ddot{\Theta}|\leq c_4 \Theta^{2}.
 \end{align}
 which is satisfied for all $k\geq 2$ with \textcolor[rgb]{0.98,0.00,0.00}{$c_4=k(k+1)T^2(T/2)^{k-4}$}.
 \end{itemize}
\end{remark}

We begin by proving first a new Carleman estimate for the problem \eqref{1eq1}-\eqref{1eq2} with one equation.
%
% ===========================================================================================
%                                 Theorem
% ===========================================================================================
%
\begin{theorem}\label{carl_1eq}
Let $T>0$ and suppose that $y_0\in H^1_{\alpha}$. Then, for all $\beta\in [\alpha,1)$
there exist two positive constants $C$ and $s_0$ such that every solution $y$ of
\eqref{1eq1}-\eqref{1eq2} satisfies for all $s\geq s_0$
\begin{align}
& \int_0^T\int_0^1\left( s\Theta(t)  x^{2\alpha-\beta}y_x^2  + s^3 \Theta^3(t)
x^{2+2\alpha-3\beta}y^2 \right)e^{2s\varphi (t,x)}\,dx dt \notag \\
& \vspace*{2cm}\leq C\left( \int_0^T\int_0^1 f^2(t,x) e^{2s\varphi (t,x)}\,dxdt+ \int_0^T s\Theta(t)y_x^2(t,1)e^{2s\varphi (t,1)}dt\right).\quad \quad \label{carl1equ}
\end{align}
\end{theorem}
\begin{proof}
For $s>0$, let us introduce the function $z:=e^{s\varphi}y$.
We have $$L_sz:=z_t+(x^\alpha z_x)_x -2sx^\alpha \varphi_xz_x -s\varphi_tz+s^2x^\alpha \varphi_x^2 z - s(x^\alpha \varphi_x)_x z  =fe^{s\varphi}.$$
Let
\begin{align*}
& L_s^+z :=  (x^\alpha z_x)_x-s\varphi_tz+s^2x^\alpha \varphi_x^2 z, \\
& L_s^-z  :=   z_t-2sx^\alpha \varphi_xz_x - s(x^\alpha \varphi_x)_x z, \\
& \hspace*{.35cm} f_s :=   fe^{s\varphi}.
\end{align*}
We have
$||f_s||_{L^2}^2=||L_s^+z+L_s^-z ||_{L^2}^2=||L_s^+z ||_{L^2}^2+||L_s^-z ||_{L^2}^2+2\langle L_s^+z,L_s^-z\rangle \geq 2\langle L_s^+z,L_s^-z\rangle$. One has $z(0,x)=z(T,x)=z_x(0,x)=z_x(T,x)=0$.
So integrating by parts one obtains
$$\langle L_s^+z,L_s^-z\rangle= -2s^2\int_0^T\int_0^1 x^\alpha \varphi_x \varphi_{tx}z^2 \,dx\,dt
+ s\int_0^T\int_0^1 x^\alpha(x^\alpha \varphi_x)_{xx}zz_x \,dx\,dt $$
$$\frac{s}{2}\int_0^T\int_0^1 \varphi_{tt}z^2 \,dx\,dt
+s\int_0^T\int_0^1 x^\alpha \left[2x^\alpha\varphi_{xx}+\alpha x^{\alpha-1}\varphi_x \right]z_x^2 \,dx\,dt$$
$$+s^3 \int_0^T\int_0^1 x^\alpha \left[2x^\alpha\varphi_{xx}+\alpha x^{\alpha-1}\varphi_x \right]\varphi_x^2z^2
$$
$$ + \int_0^T \left[
x^\alpha z_xz_t
+s^2\Theta \dot{\Theta}\psi\psi_x x^\alpha z^2
- s^3\Theta^3 x^{2\alpha}\psi_x^3 z^2
\right]_{x=0}^{x=1} dt
$$
$$    - \int_0^T \left[ \lambda(2-\beta)s\Theta x^{1-\beta}(x^\alpha z_x)^2
+ \lambda(2-\beta)(1+\alpha-\beta) s\Theta x^{2\alpha-\beta}zz_x \right]_{x=0}^{x=1} dt.
$$
It is easy to check that if $y\in H^2_{\alpha}(0,1)$ then we have also $z\in H^2_{\alpha}(0,1).$ So $x^\alpha z \in H^1(0,1)\subset L^\infty (0,1)$ by the Sobolev imbedding theorem. Then, using the facts that $z(t,0)=z(t,1)=z_t(t,0)=z_t(t,1)=0$ and $x^\alpha z_x$, $x^\alpha$, $\psi, \, \psi_x$ are bounded, we deduce that the first integral with boundary terms vanishes
and $x^{1-\beta}(x^\alpha z_x)^2|_{x=0}=0$.
On the other hand we have $\left[ x^{2\alpha-\beta}zz_x \right]_{x=0}^{x=1}=0$,
in fact it is clear that $x^{2\alpha-\beta}zz_x|_{x=1}=(x^\alpha z_x)z|_{x=1} = 0$
and since $x^\alpha z_x \in  L^\infty (0,1)$ and $z(t,0)=0$
then for each $t\in(0,T)$ we have
\begin{align}\label{mojor z,zx}
 |z_x(t,x)|\leq c x^{-\alpha} \,\, \mbox{and} \,\,\,  |z(t,x)| =|\int_0^x z_x(t,y)dy|\leq c x^{1-\alpha}.
 \end{align}
 Therefore $|x^{2\alpha-\beta}zz_x(t,x)|\leq c x^{1-\beta}$. Consequently, since $\beta<1$ we deduce $x^{2\alpha-\beta}zz_x|_{x=0}=0$.
\\
We have then
$$\underset{J_1}{\underbrace{\lambda^3(2-\beta)^3(2-2\beta+\alpha)\int_0^T\int_0^1s^3\Theta^3x^{2+2\alpha-3\beta}z^2\, dxdt}}
+ \underset{J_2}{\underbrace{\lambda(2-\beta)(2-2\beta+\alpha)\int_0^T\int_0^1s\Theta x^{2\alpha-\beta}z_x^2 \, dxdt}}$$
$$\leq  \frac12  \int_0^T\int_0^1 f^2e^{2s\varphi}\, dxdt
-\underset{J_3}{\underbrace{2\lambda^2(2-\beta)^2\int_0^T\int_0^1s^2\Theta\dot{\Theta} x^{2+\alpha-2\beta}z^2 \,dxdt}}$$
$$+ \underset{J_4}{\underbrace{\lambda(2-\beta)(1+\alpha-\beta)(\beta -\alpha)\int_0^T\int_0^1s\Theta x^{2\alpha-\beta-1}zz_x \, dxdt}}\,
+ \underset{J_5}{\underbrace{\frac{\lambda}{2}\int_0^T\int_0^1s\ddot{\Theta}(d-x^{2-\beta})z^2\,dxdt}} $$
$$+ \lambda(2-\beta)\int_0^Ts\Theta z_x^2 (t,1)dt.$$
Now we will show that $J_3, \, J_4$ and $J_5$ can be absorbed by $J_1$ and $J_2$. For this, let $\varepsilon>0$ fixed to be specified later. First, Since $\beta\geq \alpha$ and $|\Theta\dot{\Theta}|\leq C\Theta^3$ then
$$|J_3|\leq C \int_0^T\int_0^1s^2\Theta^3 x^{2+2\alpha-3\beta}z^2\, dxdt\leq \varepsilon J_1$$
for $s$ large enough.
 In the other hand for $J_4$ we have
\begin{align}
|J_4|
 & \leq \lambda(2-\beta)(1+\alpha-\beta)(\beta-\alpha)\int_0^T\int_0^1\left[\sqrt{s\Theta} x^{\alpha-\frac{\beta}{2} -1}|z|\right]\left[\sqrt{s\Theta}x^{\alpha-\frac{\beta}{2}}|z_x|\right]\, dxdt \nonumber \\
& \leq \lambda(2-\beta)(1+\alpha-\beta)(\beta-\alpha)\left(\varepsilon\int_0^T\int_0^1 s\Theta x^{2\alpha-\beta -2}z^2dxdt + \frac{1}{4\varepsilon}\int_0^T\int_0^1 s\Theta x^{2\alpha-\beta }z_x^2 dxdt\right)
 \label{2341}
\end{align}
 Now we will use the Hardy-Poincar\'e inequality \eqref{Hardy}. We have $2\alpha-\beta<1$ and we will show that $\int_0^1x^{2\alpha-\beta} z_x^2dx <+\infty$.
 Using \eqref{mojor z,zx} and the fact that $\beta<1$ we obtain,
 $$|x^{2\alpha-\beta} z_x^2|
  \, \leq \, C x^{-\beta} \in L^1(0,1).$$
We have then
 $$\int_0^1x^{2\alpha-\beta-2} z^2dx \leq C_{2\alpha-\beta} \int_0^1x^{2\alpha-\beta} z_x^2dx$$    where $C_{2\alpha-\beta}= \frac{4}{(1-2\alpha+\beta)^2}$.
 Then,  we get from \eqref{2341}
$$|J_4| \leq \lambda(2-\beta)(1+\alpha-\beta)(\beta-\alpha)\left( \varepsilon C_{2\alpha-\beta}+\frac{1}{4\varepsilon}\right)\int_0^T\int_0^1 s\Theta x^{2\alpha-\beta }z_x^2 dxdt$$
The quantity $\varepsilon C_{2\alpha-\beta}+\frac{1}{4\varepsilon}$ is minimal for $\varepsilon = \frac{1}{2.\sqrt{C_{2\alpha-\beta}}}$. For this choice we have
\begin{align*}
|J_4| \,\leq \, \lambda(2-\beta)
(1+\alpha-\beta)(\beta-\alpha)\frac{2}{1-2\alpha+\beta}
 \int_0^T\int_0^1 s\Theta x^{2\alpha-\beta }z_x^2 dxdt
\end{align*}
and for all $\beta\in [\alpha,1)$ we have
\begin{align*}
\frac{2(1+\alpha-\beta)(\beta-\alpha)}{1-2\alpha+\beta} -(2-2\beta+\alpha)=\frac{(\beta-1)(2-\alpha)}{1-2\alpha+\beta} < 0
\end{align*}
The term $J_4$ can then be absorbed by $J_2$.\\
For the last term $J_5$, since
 $|\ddot{\Theta}|\leq c_4\Theta^{2}$ and $\beta \geq \alpha$, we have by applying the Hardy-Poincar\'e inequality
\begin{align}
|J_5| & \leq \lambda d c_4 \int_0^T\int_0^1s\Theta^2 z^2dxdt
\nonumber
\\
 & = \lambda d c_4\int_0^T\int_0^1\left[\sqrt{s\Theta} x^{\alpha-\frac{\beta}{2} -1}z\right]
 \left[\sqrt{s}\Theta^{\frac32} x^{1-\alpha+\frac{\beta}{2}}z\right]\, dxdt   \nonumber\\
& \leq \lambda d c_4\int_0^T\int_0^1 \left(\varepsilon s\Theta x^{2\alpha-\beta -2}z^2 +
   \frac{1}{4\varepsilon} s\Theta^3 x^{2-2\alpha+\beta }z^2\right) dxdt  \nonumber\\
& \leq
\varepsilon \lambda d c_4C_{2\alpha-\beta} \int_0^T\int_0^1 s\Theta x^{2\alpha-\beta}z_x^2\, dxdt
 +
 C_\varepsilon \int_0^T\int_0^1 s\Theta^3 x^{2+2\alpha-3\beta }z^2\, dxdt
 \nonumber
\end{align}
Therefore by choosing $\varepsilon$ small enough, we obtain
$$\int_0^T\int_0^1s^3\Theta^3x^{2+2\alpha-3\beta}z^2\, dxdt
+ \int_0^T\int_0^1s\Theta x^{2\alpha-\beta}z_x^2 \, dxdt$$
$$\leq  C \left(  \int_0^T\int_0^1 f^2e^{2s\varphi}\, dxdt
+ \int_0^Ts\Theta z_x^2 (t,1)dt\right).$$
for $s$ large enough. So replacing $z$ by $e^{s\varphi}y$ we deduce immediately the conclusion of the theorem.
\end{proof}
%
%%%%%%%%%%%%%%%%%%%%%%%%%%%%%%%%%%%%%%%%%%%%%%%%%%%%%%%%%%%%%%%%%%%%%%%%%%%%%%%%%%%%%%
%                                      THEOREME 3.3
%%%%%%%%%%%%%%%%%%%%%%%%%%%%%%%%%%%%%%%%%%%%%%%%%%%%%%%%%%%%%%%%%%%%%%%%%%%%%%%%%%%%%%
%
%
\begin{theorem}\label{carl22_1eq}
Let $T>0$ and suppose that $y_0\in H^1_\alpha$.
Then, for all $\beta\in [\alpha,1)$ there exist two positive
constants $C$ and $s_0$ such that every solution $y$ of
\eqref{1eq1}-\eqref{1eq2} satisfies for all $s\geq s_0$
\begin{align}
& \int_0^T\int_0^1\left( s\Theta(t)  x^{2\alpha-\beta}y_x^2  + s^3 \Theta^3(t)
x^{2+2\alpha-3\beta}y^2 \right)e^{2s\varphi (t,x)}\,dx dt \notag \\
& \vspace*{2cm}\leq C\left( \int_0^T\int_0^1 f^2(t,x) e^{2s\Phi (t,x)}\,dxdt+
   \int_0^T \int_{\omega'}
 s^3\phi^3 y^2 e^{2s\Phi (t,x)} dxdt\right)
 \label{estimglobal}
\end{align}
\end{theorem}
\begin{proof}
Let us consider an arbitrary open subset $\omega'':=(a'',b'') \Subset\omega'$ and a cut-off function $\xi\in \mathcal{C}^\infty(0,1)$
such that
$$
\begin{cases} 0\leq \xi(x)\leq 1, &x\in(0,1),\\
\xi(x)=1 ,&0\leq x\leq a'',\\
\xi(x)=0 , & b''\leq x\leq 1.
\end{cases}
$$
Let $z=\xi y$ where $y$  is the solution of
\eqref{1eq1}-\eqref{1eq2}.  Then $z$ satisfies
 the following system
 \begin{align}
 & z_t-(x^\alpha z_x)_x= \xi f -\xi_x x^\alpha y_x-(x^\alpha\xi_xy)_x,  & (t,x)\in(0,T)\times(0,1), \label{firstw}\\
 &  z(t,1)=z(t,0)=0, &  t\in (0,T), \label{tirtw}
\end{align}
Therefore, applying the Carleman estimate \eqref{carl1equ} to the equation \eqref{firstw}
we obtain
\begin{align}\label{es3288}
& \int_0^T\int_0^1 [s\Theta(t)x^{2\alpha-\beta}
z_x^2(t,x)+s^3 \Theta^3(t)x^{2+2\alpha-3\beta}
z^2(t,x)]e^{2s\varphi}dx dt
\notag \\
&\leq C \int_0^T\int_0^1 [\xi^2f^2+\left(\xi_x x^\alpha
y_x+(x^\alpha \xi_xy)_x\right)^2]e^{2s\varphi}dx dt.
\end{align}
So using the definition of $\xi$ and the Cacciopoli's inequality, see  Lemma \ref{cacci}, we obtain
\begin{align}\label{es3377}
\int_0^T\int_0^1\left(\xi_x x^\alpha y_x+ (x^\alpha\xi_xy)_x\right)^2e^{2s\varphi}dxdt
&\leq C \int_0^T\int_{\omega''}[y^2+y_x^2]e^{2s\varphi}dx dt\nonumber \\
& \leq C\int_0^T\int_{\omega'} y^2e^{2s\varphi}dx dt.
\end{align}
and
\begin{align}
\int_0^T\int_0^1s\Theta x^{2\alpha-\beta}\xi^2 y_x^2 e^{2s\varphi}dxdt\leq 2 \int_0^T\int_0^1 s\Theta x^{2\alpha-\beta}z_x^2e^{2s\varphi}dxdt + 2\int_0^T\int_{\omega'}s\Theta y^2e^{2s\varphi}dxdt
\end{align}
Thus from \eqref{es3288}-\eqref{es3377} and the definition of $\xi$ we deduce the following estimate
\begin{align}\label{estimleft}
& \int_0^T\int_0^1 [s\Theta(t)x^{2\alpha-\beta}
\xi^2y_x^2(t,x)+s^3 \Theta^3(t)x^{2+2\alpha-3\beta}
\xi^2y^2(t,x)]e^{2s\varphi}dx dt
\notag \\
&\leq C \left( \int_0^T\int_0^1 \xi^2 f^2e^{2s\varphi}dxdt + \int_0^T\int_{\omega'} s\Theta y^2e^{2s\varphi}dxdt\right).
\end{align}
On $(a',1)$ the equation \eqref{1eq1} is uniformly parabolic
hence, one can use the following Carleman estimate which is a consequence of
(\cite{Fursikov}, Lemma 1.2) established by Fursikov and Imanuvilov.
\begin{proposition}
Consider the nondegenerate linear problem
$$\begin{cases}
 v_t-(x^\alpha v_x)_x=f \in L^2((0,T)\times(a',1)), \\
 v(t,a')=v(t,1)=0,\,\,\, t\in (0,T), & \label{alban-Can2}
\end{cases}$$
Then, there exists a constant $\rho_0 >0$ such that for all $\rho\geq\rho_0$ there exists $s_0(\rho)>0$ such that for each $s\geq s_0(\rho)$ the solution $v$ of the last problem
satisfy the following estimate:
\begin{align}\label{classicFursikov}
 & \int_0^T\int_{a'}^1 (s\phi v_x^2+
s^3 \phi^3 v^2)e^{2s\Phi}dx dt  \notag \\
& \leq C \left(\int_0^T\int_{a'}^1 f^2e^{2s\Phi}dx dt  + \int_0^T\int_{\omega'} s^3 \phi^3 v^2e^{2s\Phi}dx dt\right)
 \end{align}
where the functions
 $\Phi$ and $\phi$ are defined in Theorem \ref{carl22_1eq}.
\end{proposition}
\begin{remark}
 The last estimate was showed  in  \cite{Fursikov} for $\Theta(t)=\frac{1}{t(T-t)}$
but by careful examination of the proof one can see easily that it remains valid for all $\Theta\in C^2(0,T)$ satisfying
\eqref{majorTheta1}, see Remark \ref{rk1}.
\end{remark}
To achieve the proof of the Theorem \ref{carl01}, let $Z:=\zeta y,$ where the function
$\zeta$ is defined as $\zeta=1-\xi.$ Then $Z$ is a solution of the following problem
\begin{align*}
& Z_t-(x^\alpha Z_x)_x=\zeta f-\zeta_x x^\alpha
 y_x-(x^\alpha\zeta_x y)_x, & (t,x)\in(0,T)\times(a',1),
 %\label{secondz}
 \\
& Z(t,1)=Z(t,a'
)=0 ,&  t\in (0,T),
% \label{tirtz1}
\end{align*}
Applying the classical Carleman estimate \eqref{classicFursikov},
it follows that for $s$ large enough
\begin{align*}
  & \int_0^T\int_{0}^1 (s\phi Z_x^2  +
s^3 \phi^3 Z^2)e^{2s\Phi}dx dt      \\
& \leq  C\left(\int_0^T\int_{0}^1 \left[\zeta f+\zeta_x x^\alpha
 y_x+(x^\alpha \zeta_x y)_x\right]^2e^{2s\Phi}dx dt
 + \int_0^T\int_{\omega'} s^3 \phi^3 Z^2e^{2s\Phi}dx dt\right)     \\
& \leq  C\left(\int_0^T\int_{0}^1 \zeta^2 f^2e^{2s\Phi}dx dt
  + \int_0^T\int_{\omega''}[y^2+y_x^2] e^{2s\Phi}dxdt\right.
\left.+ \int_0^T\int_{\omega'} s^3 \phi^3 Z^2e^{2s\Phi}dx dt \right)
\end{align*}
Therefore, using the Caccioppoli inequality and the definitions of $Z$ and $\zeta$ we deduce
\begin{align}
  & \int_0^T\int_{0}^1 (s\phi \zeta^2y_x^2  +
s^3 \phi^3 \zeta^2y^2)e^{2s\Phi}dx dt     \nonumber \\
& \leq  C\left(\int_0^T\int_{0}^1 \zeta^2 f^2e^{2s\Phi}dx dt
+ \int_0^T\int_{\omega'} s^3 \phi^3 y^2e^{2s\Phi}dx dt \right)\label{estimright}
\end{align}
Thanks to \eqref{comparaison phi}  there exists a constant $c>0$ such that for all $(t,x)\in[0,T]\times(a',1)$ one has
\begin{align}\label{compare varphi and Phi}
\Theta x^{2\alpha-\beta}e^{2s\varphi(t,x)}\leq c
\phi e^{2s\Phi(t,x)} \,\mbox{and} \,\,
\Theta^3 x^{2+2\alpha-3\beta}e^{2s\varphi(t,x)}\leq c \phi^3e^{2s\Phi(t,x)}
\end{align}
Then, using \eqref{estimleft}, \eqref{estimright}, \eqref{comparaison phi}, \eqref{majorTheta1} and the fact that $1/2 \leq \xi^2+\zeta^2\leq 1$
 we obtain the global estimate
\begin{align}
& \int_0^T\int_0^1\left( s\Theta(t)  x^{2\alpha-\beta}y_x^2  + s^3 \Theta^3(t)
x^{2+2\alpha-3\beta}y^2 \right)e^{2s\varphi (t,x)}\,dx dt \notag \\
& \vspace*{2cm}\leq C\left( \int_0^T\int_0^1 f^2(t,x) e^{2s\Phi (t,x)}\,dxdt+
   \int_0^T \int_{\omega'}  s^3\phi^3 y^2 e^{2s\Phi (t,x)} dxdt\right)
 \label{estimglobal1}
\end{align}
This ends the proof of Theorem \ref{carl22_1eq}.
\end{proof}
The estimate in Theorem \ref{carl22_1eq} was obtained for regular initial data.
 By density we deduce the following result for the general case: $y_0 \in L^2(0,1)$.
 %which is sufficient to obtain observability inequality
 %
 %%%%%%%%%%%%%%%%%%%%%%%%%%%%%%%%%%%%%%%%%%%%%%%%%%%%%%%%%%%%%%%%%%%%%%%%%%%%%%%%%%%%%%%
 %
 %                             Corollary
 %
 %%%%%%%%%%%%%%%%%%%%%%%%%%%%%%%%%%%%%%%%%%%%%%%%%%%%%%%%%%%%%%%%%%%%%%%%%%%%%%%%%%%%%%%
 \begin{corollary}
Let $T>0$ be given. Let $\beta\in [\alpha,1)$ and $\mu\geq max(0,2+2\alpha-3\beta)$. Then there exist two positive
constants $C$ and $s_0$ such that every solution $y$ of
\eqref{1eq1}-\eqref{1eq2} satisfies for all $s\geq s_0$
\begin{align}
& \int_0^T\int_0^1\left( s\Theta  x^{\alpha}y_x^2  + s^3 \Theta^3
x^{\mu}y^2 \right)e^{2s\varphi (t,x)}\,dx dt \notag \\
& \vspace*{2cm}\leq C\left( \int_0^T\int_0^1 f^2(t,x) e^{2s\Phi (t,x)}\,dxdt+
   \int_0^T \int_{\omega'}
 s^3\phi^3
y^2
e^{2s\Phi (t,x)}
dxdt\right)
 \label{estimgloballll}
\end{align}
 \end{corollary}
\begin{proof}
Let $y_0\in L^2(0,1).$ By the density of $H^1_\alpha(0,1)$ in $L^2(0,1)$, there exist a set $(y_0^n)_n$ in $H^1_\alpha(0,1)$ which converges to  $y_0$.
Let $y^n$ the unique solution in the space
 $Z_T:=C\left( [0,T],L^2(0,1) \right)
\cap L^2\left(0,T;H^1_{\alpha}\right)$
of the problem \eqref{1eq1}-\eqref{1eq2} associated to the initial data $y_0^n$. As in \eqref{estiminitialdata}
 one has for a constant $C_T>0$
 $$\|(y^m-y^n)(t)\|_{Z_T}:= \sup_{[0,T]}\|(y^m-y^n)(t)\|^2_{L^2} + \int_0^T \| x^{\frac{\alpha}{2}}(y^m-y^n)_x \|^2_{L^2} \, dt
\leq C_T \|y^m_0-y^n_0\|^2_{L^2}.$$
Therefore the set $(y^n)_n$ has a limit $y$ in the Banach space $Z_T$. Using classical argument in semigroup theory it is easy to show that $y$ is the solution of the problem \eqref{1eq1}-\eqref{1eq2} associated to the initial data $y_0$.
 On the other hand since
  $x^{\alpha}\leq x^{2\alpha-\beta}$ and $x^{\mu}\leq x^{2+2\alpha-3\beta}$ on (0,1) then we deduce from Theorem \ref{carl22_1eq} the estimate
\begin{align*}
& \int_0^T\int_0^1\left( s\Theta  x^{\alpha}|y^n_x|^2  + s^3 \Theta^3
x^{\mu}|y^n|^2 \right)e^{2s\varphi (t,x)}\,dx dt \notag \\
& \vspace*{2cm}\leq C\left( \int_0^T\int_0^1 f^2(t,x) e^{2s\Phi (t,x)}\,dxdt+
   \int_0^T \int_{\omega'}  s^3\phi^3|y^n|^2 e^{2s\Phi (t,x)}dxdt\right)
\end{align*}
 And since $s\Theta e^{2s\varphi}$, $s^3\Theta^3 e^{2s\varphi}x^{\mu}$ and $s^3\phi^3 e^{2s\Phi}$ are bounded then one can pass to the limit and get the desired estimate.
\end{proof}
%
%%%%%%%%%%%%%%%%%%%%%%%%%%%%%%%%%%%%%%%%%%%%%%%%%%%%%%%%%%%%%%%%%%%%%%%%%%%%%%%%%%%%%%%%%%%%%%%%%%%%%%
%
%                                     SYSTEM
%
%%%%%%%%%%%%%%%%%%%%%%%%%%%%%%%%%%%%%%%%%%%%%%%%%%%%%%%%%%%%%%%%%%%%%%%%%%%%%%%%%%%%%%%%%%%%%%%%%%%%%%
%
For the coupled system \eqref{adjfirst}-\eqref{adjend} we prove first an intermediate important result which could be used to show the null controllability for a coupled system with two control forces
\begin{theorem}\label{carl01}
Let $T>0$ and $(\alpha_1,\alpha_2)\in (0,1)\times(0,1)$ be given and suppose that $y_0\in H_\alpha^1$. Then for all
$\beta \in [max(\alpha_1,\alpha_2),1[$
there exist two positive
constants $C$ and $s_0$ such that every solution $(U,V)$ of
\eqref{adjfirst}-\eqref{adjend} satisfies
\begin{align}\label{Carl2degenr}
& \int_0^T\int_0^1 s\Theta(t) \left[ x^{2\alpha_1-\beta}U_x^2(t,x)+x^{2\alpha_2-\beta}V_x^2(t,x)\right]e^{2s\varphi (t,x)}\,dx dt \notag \\
& +\int_0^T\int_0^1 s^3 \Theta^3(t)\left[
x^{2+2\alpha_1-3\beta}U^2(t,x)+x^{2+2\alpha_2-3\beta}V^2(t,x)\right] e^{2s\varphi (t,x)}\,dx dt \notag \\
& \vspace*{2cm}\leq C\int_0^T\int_{\omega'} s^3
\Theta^3 \left[U^2(t,x)+V^2(t,x)\right]e^{2s\Phi(t,x)}
\,dxdt \quad \quad \mbox{for all}\,\,s\geq s_0.
\end{align}
 \end{theorem}
\begin{proof}
Since $U$ is solution of the problem
\begin{align*}
& U_{t} -(x^{\alpha_1}U_x)_x =- b_{11}(t,x) U-b_{21}(t,x) V, &  (t,x)\in (0,T)\times(0,1), \\
& U(t,1)=U(t,0)=0,\, & t\in (0,T),\\
& U(0,x)=U_0(x), & x\in (0,1),
\end{align*}
then applying the estimate \eqref{estimleft} to this system we obtain
\begin{align}\label{estimU}
& \int_0^T\int_0^1 [s\Theta(t)x^{2\alpha_1-\beta}
\xi^2U_x^2(t,x)+s^3 \Theta^3(t)x^{2+2\alpha_1-3\beta}
\xi^2U^2(t,x)]e^{2s\varphi}dx dt
\notag \\
&\leq \overline{C} \int_0^T\int_0^1 \xi^2 (b_{11}^2U^2+b_{21}^2V^2)e^{2s\varphi}dxdt + C\int_0^T\int_{\omega'}s\Theta U^2e^{2s\varphi}dxdt.
\end{align}
Using the Hardy-Poincar\'e inequality \eqref{Hardy} one has for  $s$ large enough
\begin{align*}
 \int_0^T\int_0^1 b_{11}^2 \xi^2U^2e^{2s\varphi}dxdt &\leq C
  \int_0^T\int_0^1\left[ x^{\alpha_1 -\frac{\beta}{2} -1}\xi Ue^{s\varphi}\right]
 \left[ x^{1-\alpha_1+\frac{\beta}{2}}\xi Ue^{s\varphi}\right]\, dxdt  \nonumber \\
& \leq C \int_0^T\int_0^1 \left( x^{2\alpha_1-\beta -2}\xi^2U^2 +
  x^{2-2\alpha_1+\beta }\xi^2U^2\right)e^{2s\varphi} dxdt \nonumber \\
&  \leq C\int_0^T\int_0^1  x^{2\alpha_1-\beta }(\xi Ue^{s\varphi})_x^2 dxdt
+
C\int_0^T\int_0^1  x^{2-2\alpha_1+\beta }\xi^2U^2e^{2s\varphi} dxdt \nonumber \\
& \leq C \int_0^T\int_0^1  \left( x^{2\alpha_1-\beta }\xi^2 U_x^2+x^{2\alpha_1-\beta }\xi_x^2 U^2+s^2\Theta^2x^{2+2\alpha_1-3\beta}\xi^2U^2\right)e^{2s\varphi}dxdt  \nonumber \\
&\hspace*{.5cm} + C \int_0^T\int_0^1  x^{2-2\alpha_1+\beta }\xi^2U^2e^{2s\varphi} dxdt.
\end{align*}
So since $\beta\geq \alpha$, $\xi_x$ is supported in $\omega'$ and $\Theta$ is bounded below then for $s$ large enough we have
\begin{align}\label{b11U}
\bar{C} \int_0^T\int_0^1 b_{11}^2 \xi^2U^2e^{2s\varphi}dxdt \, \leq \,\,\,  & \frac14  \int_0^T\int_0^1 [s\Theta(t)x^{2\alpha_1-\beta} \xi^2 U_x^2+ s^3 \Theta^3(t)x^{2+2\alpha_1-3\beta}
\xi^2U^2]e^{2s\varphi}dx dt
\nonumber \\
& + C\int_0^T\int_{\omega'} U^2e^{2s\varphi}dx dt
\end{align}
Similarly, for $s$ large enough we have
\begin{align}\label{b21V}
\bar{C} \int_0^T\int_0^1 b_{21}^2 \xi^2V^2e^{2s\varphi}dxdt \, \leq \,\,\,  & \frac14  \int_0^T\int_0^1 [s\Theta(t)x^{2\alpha_2-\beta} \xi^2 V_x^2+ s^3 \Theta^3(t)x^{2+2\alpha_2-3\beta}
\xi^2V^2]e^{2s\varphi}dx dt
\nonumber \\
& + C\int_0^T\int_{\omega'} V^2e^{2s\varphi}dx dt
\end{align}
Combining \eqref{estimU}, \eqref{b11U} and \eqref{b21V} we deduce the estimate
\begin{align}\label{estimU2}
 \int_0^T\int_0^1 [s\Theta(t)x^{2\alpha_1-\beta}
\xi^2U_x^2(t,x) & +s^3 \Theta^3(t)x^{2+2\alpha_1-3\beta}
\xi^2U^2(t,x)]e^{2s\varphi}dx dt
\notag \\
&  \leq \,\, \frac14  \int_0^T\int_0^1 [s\Theta(t)x^{2\alpha_1-\beta} \xi^2 U_x^2+ s^3 \Theta^3(t)x^{2+2\alpha_1-3\beta}
\xi^2U^2]e^{2s\varphi}dx dt
\nonumber \\
& \,\,\,\,\,\,\,+ \frac14  \int_0^T\int_0^1 [s\Theta(t)x^{2\alpha_2-\beta} \xi^2 V_x^2+ s^3 \Theta^3(t)x^{2+2\alpha_2-3\beta}
\xi^2V^2]e^{2s\varphi}dx dt
\nonumber \\
& \,\,\,\,\,\,\, + C\int_0^T\int_{\omega'} s\Theta(U^2+V^2)e^{2s\varphi}dx dt.
\end{align}
For the second component, Arguing as before we have for $s$ large enough
\begin{align}\label{estimV2}
 \int_0^T\int_0^1 [s\Theta(t)x^{2\alpha_2-\beta}
\xi^2V_x^2(t,x) & +s^3 \Theta^3(t)x^{2+2\alpha_2-3\beta}
\xi^2V^2(t,x)]e^{2s\varphi}dx dt
\notag \\
&  \leq \,\, \frac14  \int_0^T\int_0^1 [s\Theta(t)x^{2\alpha_2-\beta} \xi^2 V_x^2+ s^3 \Theta^3(t)x^{2+2\alpha_2-3\beta}
\xi^2V^2]e^{2s\varphi}dx dt
\nonumber \\
& \,\,\,\,\,\,\,+ \frac14  \int_0^T\int_0^1 [s\Theta(t)x^{2\alpha_1-\beta} \xi^2 U_x^2+ s^3 \Theta^3(t)x^{2+2\alpha_1-3\beta}
\xi^2U^2]e^{2s\varphi}dx dt
\nonumber \\
& \,\,\,\,\,\,\, + C\int_0^T\int_{\omega'} s\Theta(U^2+V^2)e^{2s\varphi}dx dt.
\end{align}
Therefore, from \eqref{estimU2} and \eqref{estimV2} we deduce the estimate
\begin{align}\label{estimUV1}
 \int_0^T\int_0^1 [s\Theta(t)x^{2\alpha_1-\beta}
\xi^2U_x^2(t,x) & +s^3 \Theta^3(t)x^{2+2\alpha_1-3\beta}
\xi^2U^2(t,x)]e^{2s\varphi}dx dt
\notag \\
+ \int_0^T\int_0^1 [s\Theta(t)x^{2\alpha_2-\beta}
\xi^2V_x^2(t,x) & +s^3 \Theta^3(t)x^{2+2\alpha_2-3\beta}
\xi^2V^2(t,x)]e^{2s\varphi}dx dt
\notag \\
\leq C\int_0^T\int_{\omega'} s\Theta(U^2+V^2)e^{2s\varphi}dx dt.
\end{align}
This gives an estimate on $(0,a')$.
 As above, to obtain an estimate on $(a',1)$, we apply \eqref{estimright} to each equation of the system \eqref{adjfirst}-\eqref{adjend},  we use Hardy-Poincar\'e inequality and we obtain the estimate
\begin{align}\label{estimUV2}
 \int_0^T\int_0^1 [s\phi
\zeta^2(U_x^2+V_x^2)  +s^3 \phi^3\zeta^2(U^2+V^2)]e^{2s\Phi}dx dt
\notag \\
\leq C\int_0^T\int_{\omega'}s^3\phi^3 (U^2+V^2)e^{2s\Phi}dx dt.
\end{align}
Consequently, using \eqref{estimUV1}, \eqref{estimUV2} and \eqref{compare varphi and Phi} we deduce the global estimate
\begin{align*}
& \int_0^T\int_0^1 s\Theta(t) \left[ x^{2\alpha_1-\beta}U_x^2(t,x)+x^{2\alpha_2-\beta}V_x^2(t,x)\right]e^{2s\varphi}\,dx dt \notag \\
& +\int_0^T\int_0^1 s^3 \Theta^3(t)\left[
x^{2+2\alpha_1-3\beta}U^2(t,x)+x^{2+2\alpha_2-3\beta}V^2(t,x)\right] e^{2s\varphi}\,dx dt \notag \\
& \vspace*{2cm}\leq C\int_0^T\int_{\omega'} s^3
\Theta^3
 \left[U^2(t,x)+V^2(t,x)\right]e^{2s\Phi} \,dxdt.
\end{align*}
This ends the proof.
\end{proof}
As above, using density argument we deduce the following result for the general case: $U_0, \, V_0 \in L^2(0,1)$.
\begin{corollary}\label{??????}
Let $T>0$ and $(\alpha_1,\alpha_2)\in (0,1)\times(0,1)$ be given.
Let $\beta \in [max(\alpha_1,\alpha_2),1[$ and $\mu_i \geq max(0,2+2\alpha_i-3\beta)$.
Then, there exist two positive
constants $C$ and $s_0$ such that every solution $(U,V)$ of
\eqref{adjfirst}-\eqref{adjend} satisfies
\begin{align}
& \int_0^T\int_0^1 s\Theta(t) \left[ x^{\alpha_1}U_x^2(t,x)+x^{\alpha_2}V_x^2(t,x)\right]e^{2s\varphi (t,x)}\,dx dt \notag \\
& +\int_0^T\int_0^1 s^3 \Theta^3(t)\left[
x^{\mu_1}U^2(t,x)+x^{\mu_2}V^2(t,x)\right] e^{2s\varphi (t,x)}\,dx dt \notag \\
& \vspace*{2cm}\leq C\int_0^T\int_{\omega'} s^3
\Theta^3(t) \left[U^2(t,x)+V^2(t,x)\right]e^{2s\Phi(t,x)}
\,dxdt \quad \quad \mbox{for all}\,\,s\geq s_0. \label{Carl2degenrL2}
\end{align}
 \end{corollary}
 %
%%%%%%%%%%%%%%%%%%%%%%%%%%%%%%%%%%%%%%%%%%%%%%%%%%%%%%%%%%%%%%%%%%%%%%%%%%%%%%%%%
%             SYSTEM  ONE  FORCE
%
%%%%%%%%%%%%%%%%%%%%%%%%%%%%%%%%%%%%%%%%%%%%%%%%%%%%%%%%%%%%%%%%%%%%%%%%%%%%%%%%%%%%
%
To study of the null-controllability of the system \eqref{eq:1}-\eqref{eq:2}
% under one force, for instance  $h_2=0$,
we need to show the following Carleman estimate.
\begin{theorem} \label{carl1}
Let $T>0$ be given. Assume moreover that
\begin{align}
b_{21}\geq \mu \,\, \mbox{on} \, [0,T]\times \omega_1 \quad \mbox{for some}\,\, \omega_1\Subset \omega\,\, \mbox{and}\,\, \mu>0. \label{b2supGama}
\end{align}
 Then there exist two positive constants $C$ and
$s_0$ such that, every solution $(U,V)$ of
\eqref{adjfirst}-\eqref{adjend} satisfies for all $s\geq s_0$ the estimates
\begin{align}
& \int_0^T\int_0^1 s\Theta(t)\left[x^{2\alpha_1-\beta}U_x^2(t,x)+x^{2\alpha_2-\beta}V_x^2(t,x))\right]e^{2s\varphi (t,x)}\,dx dt\notag \\
& + \int_0^T\int_0^1 s^3 \Theta^3(t)
\left[x^{2+2\alpha_1-3\beta}U^2(t,x)+x^{2+2\alpha_2-3\beta}V^2(t,x)\right]e^{2s\varphi (t,x)}\,dx dt\notag
\\
& \leq C\int_0^T\int_\omega  U^2(t,x)\,dx dt. \label{carleman2}
\end{align}
\end{theorem}
\begin{remark}
The assumption \eqref{b2supGama} can be replaced by \\
\hspace*{4,3cm} $b_{21}\leq -\mu$ \,\, on \, $[0,T]\times \omega_1$ \quad for some \, $\omega_1\Subset \omega$\, and\, $\mu>0$.
\end{remark}
Theorem \ref{carl1} is a consequence of Theorem \ref{carl01} applied to $\omega_1$ and the following lemma, see also the proofs of (\cite{Ca-Te}, Theorem 3.2), \cite{Liu} and \cite{bahm}.
\begin{lemma}\label{lemma} Suppose moreover that \eqref{b2supGama} holds. Then
for all $\varepsilon >0$ there exists a positive constant $C_\varepsilon >0$ such that every solution $(U,V)$ of
\eqref{adjfirst}-\eqref{adjend} satisfies
\begin{align}\label{estilemma}\int_0^T\int_{\omega_1} s^3\Theta^3 V^2
e^{2s\Phi}dx dt& \leq \varepsilon J(V)+C_\varepsilon\int_0^T\int_{\omega} U^2 dx dt,
\end{align}
\end{lemma}
where $\omega_1$ is defined in \eqref{b2supGama} and
$$J(V):=\int_0^T\int_0^1 \left[s\Theta(t) x^{2\alpha_2-\beta}V_x^2+
 s^3 \Theta^3(t) x^{2+2\alpha_2-3\beta}V^2\right]e^{2s\varphi (t,x)}\,dx dt
$$
\begin{proof}
Let $\chi\in\mathcal{C}^\infty(0,1)$ such that $supp \,\chi \subset
\omega$ and $\chi \equiv 1$ on $\omega_1$. Multiplying the equation
\eqref{adjfirst} by $s^3\Theta^3\chi e^{2s\Phi}  V$ and integrating, we obtain
\begin{align}\label{lem1forc1}
\int_0^T\int_0^1 \chi b_{21} s^3\Theta^3 e^{2s\Phi}  V^2 dx dt
  = & -\int_0^T\int_0^1 \chi s^3\Theta^3 e^{2s\Phi}VU_t\,dxdt + \int_0^T\int_0^1 \chi s^3\Theta^3 e^{2s\Phi}V(x^{\alpha_1}U_x)_x dxdt \notag \\
   & - \int_0^T\int_0^1 \chi b_{11}s^3\Theta^3 e^{2s\Phi}UV dxdt
\end{align}
Integrating by parts and using the equation \eqref{21}, we obtain
\begin{align}\label{lem1forc2}
\int_0^T\int_0^1\chi s^3\Theta^3 e^{2s\Phi}VU_t\,dxdt =
     &    \int_0^T\int_0^1 \chi x^{\alpha_2} s^3\Theta^3 e^{2s\Phi}U_xV_x\,dxdt
       + \int_0^T\int_0^1 x^{\alpha_2} s^3\Theta^3(\chi e^{2s\Phi})_x UV_x \,dxdt \notag \\
     &  + \int_0^T\int_0^1 \chi b_{12} s^3\Theta^3 e^{2s\Phi}U^2\,dxdt
       + \int_0^T\int_0^1 \chi b_{22} s^3\Theta^3 e^{2s\Phi}UV \,dxdt   \notag \\
     &  -\int_0^T\int_0^1 \chi s^3\left(\Theta^3 e^{2s\Phi}\right)_t UV \,dxdt ,
\end{align}
and
\begin{align}\label{lem1forc3}
& \int_0^T\int_0^1\chi s^3\Theta^3 e^{2s\Phi}V(x^{\alpha_1}U_x)_x\,dxdt =
       - \int_0^T\int_0^1 x^{\alpha_1}\chi s^3\Theta^3 e^{2s\Phi}U_xV_x\,dxdt \notag \\
& \hspace*{1cm}       + \int_0^T\int_0^1 s^3\Theta^3  x^{\alpha_1}(\chi e^{2s\Phi})_x UV_x \,dxdt
     + \int_0^T\int_0^1 s^3\Theta^3 (x^{\alpha_1}(\chi e^{2s\Phi})_x)_x UV\,dxdt.
\end{align}
So combining the identities \eqref{lem1forc1}-\eqref{lem1forc3}, we get
\begin{align}
& \int_0^T\int_0^1 b_{21} \chi s^3\Theta^3 e^{2s\Phi}  V^2 dx dt
\, = \,\, - \,\, \underset{I_1}{\underbrace{ \int_0^T\int_0^1 (x^{\alpha_1}+x^{\alpha_2})\chi s^3\Theta^3 e^{2s\Phi}U_xV_x\,dxdt}} \notag \\
& \hspace*{1.5cm}      + \, \underset{I_2}{\underbrace{ \int_0^T\int_0^1 (x^{\alpha_1}-x^{\alpha_2}) s^3\Theta^3(\chi e^{2s\Phi})_x UV_x \,dxdt}} \, - \, \underset{I_3}{\underbrace{  \int_0^T\int_0^1 b_{12}\chi s^3\Theta^3 e^{2s\Phi}U^2\,dxdt}}\notag \\
&  \hspace*{1.5cm}      +\, \underset{I_4}{\underbrace{  \int_0^T\int_0^1 \left[ s^3\chi\left(\Theta^3 e^{2s\Phi}\right)_t
+s^3\Theta^3(x^{\alpha_1}(\chi e^{2s\Phi})_x)_x
-(b_{11}+b_{22})\chi s^3\Theta^3 e^{2s\Phi}
           \right]UV \,dxdt}}.  \label{lem1forc4}
\end{align}
Now we estimate the  integrals $I_1,\, I_2,\, I_3 $ and $I_4.$
We have
\begin{align}
& \left|\int_0^T\int_0^1 x^{\alpha_i}\chi s^3\Theta^3 e^{2s\Phi}U_xV_x\,dxdt\right|   =
\left| \int_0^T\int_0^1 \left[s^{\frac12}\Theta^{\frac12} x^{\alpha_2-\frac{\beta}{2}}e^{s\varphi }V_x\right]\left[s^{\frac52}\Theta^{\frac52} \chi x^{\alpha_i-\alpha_2+\frac{\beta}{2}}e^{s(2\Phi-\varphi)} U_x\right]dx dt\right|\notag \\
& \leq \varepsilon \int_0^T\int_0^1 s\Theta x^{2\alpha_2-\beta}e^{2s\varphi }V^2_x\,dxdt
   + \frac{1}{4\varepsilon} \underset{K}{\underbrace{\int_0^T\int_0^1s^{5}\Theta^{5} \chi^2 x^{2\alpha_i-2\alpha_2+\beta}e^{2s(2\Phi-\varphi)} U_x^2 dx dt}}   . \label{lem1forc5}
\end{align}
The last integral $K$ should be estimated by an integral in $U^2$. For this,  we multiply the equation  \eqref{adjfirst}
 by $
s^{5}\Theta^{5}\chi^2x^{\mu}e^{2s(2\Phi-\varphi)}U$ where $\mu := 2\alpha_i-\alpha_1-2\alpha_2+\beta$, we integrate by parts and we obtain
\begin{align*}
K   = & \,\, \frac12 \underset{K_1}{\underbrace{\int_0^T\int_0^1s^5\left(\Theta^5e^{2s(2\Phi-\varphi)}\right)_t
\chi^2x^\mu  U^2 dx dt}}  \\
 &  +\frac12 \underset{K_2}{\underbrace{\int_0^T\int_0^1 s^5\Theta^5 (x^{\alpha_1}(\chi^2x^\mu e^{2s(2\Phi-\varphi)})_x)_xU^2 \, dxdt}}
 -\underset{K_3}{\underbrace{\int_0^T\int_0^1 b_{11}s^5\Theta^5 \chi^2x^\mu e^{2s(2\Phi-\varphi)}U^2 \, dxdt}} \\
 &  -\underset{K_4}{\underbrace{\int_0^T\int_0^1 b_{21} s^5\Theta^5 \chi^2x^\mu e^{2s(2\Phi-\varphi)}UV \, dxdt}}.
\end{align*}
Since $|\Theta^{'}|\leq C\Theta^2$ and  $supp \chi \subset \omega$
we have for $i\in \{1,2,3\}$
\begin{align*}
 & |K_i| \leq  C\int_0^T\int_\omega s^7\Theta^7 e^{2s(2\Phi-\varphi)} U^2 dx dt,
 \end{align*}
 For $i=4$ we have
 \begin{align*}
  |K_4| & =
 \int_0^T\int_0^1 [s^{\frac32}\Theta^{\frac32}x^{1+\alpha_2-\frac{3\beta}{2}} e^{s\varphi }V]
 [s^{\frac{7}{2}}\Theta^{\frac{7}{2}}b_{21}\chi^2x^{\mu-1-\alpha_2+\frac{3\beta}{2}}
 e^{s(4\Phi-3\varphi)}U ]dx dt  \notag \\
        & \leq  \varepsilon^2 \int_0^T\int_0^1
         s^3\Theta^3 x^{2+2\alpha_2-3\beta} e^{2s\varphi }V^2 dxdt
 + C_\varepsilon \int_0^T\int_\omega
        s^7\Theta^7
 e^{2s(4\Phi-3\varphi)} U^2 dx dt.
 \end{align*}
So, thanks to \eqref{comparaison phi} we have
\begin{align}
|K|\leq & \,\,\varepsilon^2 \int_0^T\int_0^1 s^3\Theta^3 x^{2+2\alpha_2-3\beta} e^{2s\varphi }V^2 dxdt
  + C_\varepsilon\int_0^T\int_\omega   U^2 dx dt. \label{lem1forc11}
\end{align}
From \eqref{lem1forc5}-\eqref{lem1forc11} we deduce the estimate
\begin{align}
|I_1|\leq \,\, 2\varepsilon \int_0^T\int_0^1 s\Theta x^{2\alpha_2-\beta}e^{2s\varphi }V^2_x\,dxdt
+ \frac\varepsilon2 \int_0^T\int_0^1 s^3\Theta^3 x^{2+2\alpha_2-3\beta} e^{2s\varphi }V^2 dxdt
  + C_\varepsilon \int_0^T\int_\omega   U^2 dx dt. \label{lem1forc12}
\end{align}
Similarly we have
\begin{align}
 |I_2| \leq  & \,\, C \int_0^T\int_{\omega} s^4\Theta^4 |UV_x|e^{2s\Phi} dxdt \nonumber
 \\
  \leq & \,\, \varepsilon \int_0^T\int_0^1 s\Theta x^{2\alpha_2-\beta}e^{2s\varphi }V^2_x\,dxdt +
C_\varepsilon \int_0^T\int_\omega   U^2 dx dt, \label{lem1forc13} \\
|I_3| \leq & \,\, C \int_0^T\int_\omega U^2 dx dt, \label{lem1forc15}\\
|I_4|  \leq  & \,\, C \int_0^T\int_{\omega} s^6\Theta^6 |UV|e^{2s\Phi} dxdt \nonumber
\\
\leq & \,\, \varepsilon \int_0^T\int_0^1 s^3\Theta^3 x^{2+2\alpha_2-3\beta}e^{2s\varphi }V^2\,dxdt +
C_{\varepsilon} \int_0^T\int_\omega U^2 dx dt. \label{lem1forc14}
\end{align}
Consequently, from the estimates \eqref{lem1forc12}-\eqref{lem1forc14}, we conclude that
\begin{align*}
\int_0^T\int_0^1 b_{21} \chi e^{2s\varphi }  V^2 dx dt \leq 3\varepsilon J(V) + C_{\varepsilon} \int_0^T\int_\omega U^2 dx dt.
\end{align*}
Finally, since $\chi \equiv 1$ on $\omega_1$, then using \eqref{b2supGama} we achieve the claim.
\end{proof}
As above, using a density argument we deduce the following result for the general case: $U_0,\, V_0\in L^2(0,1)$.
\begin{corollary} \label{corol1forc}
 Let $T>0$ be given. Assume moreover that \eqref{b2supGama} holds.
Let $(\alpha_1,\alpha_2)\in (0,1)\times(0,1)$,  $\beta \in [max(\alpha_1,\alpha_2),1[$ and $\mu_i \geq max(0,2+2\alpha_i-3\beta)$.
 Then, there exist two positive constants $C$ and
$s_0$ such that, every solution $(U,V)$ of
\eqref{adjfirst}-\eqref{adjend} satisfies, for all $s\geq s_0$ the estimates
\begin{align}
& \int_0^T\int_0^1 s\Theta(t)\left[x^{\alpha_1}U_x^2(t,x)+x^{\alpha_2}V_x^2(t,x))\right]e^{2s\varphi (t,x)}\,dx dt\notag \\
& + \int_0^T\int_0^1 s^3 \Theta^3(t)
\left[x^{\mu_1}U^2(t,x)+x^{\mu_2}V^2(t,x))\right]e^{2s\varphi (t,x)}\,dx dt\notag
\\
& \leq C\int_0^T\int_\omega  U^2(t,x)\,dx dt. \label{carlemanonef}
\end{align}
\end{corollary}
%
%
%==================================================================
%
%              OBSEVABILITY    AND CONTROLLABILITY
%
%==================================================================
%
%
\section{Observability and null controllability of linear systems}
% In this section we prove,
As a consequence of the Carleman
estimates established in the above section, we prove first a observability
inequality for the adjoint problem \eqref{adjfirst}-\eqref{adjend}
of problem \eqref{eq:1}-\eqref{eq:2}.
\begin{theorem}\label{observatos}
Let $T>0$ be given. Assume that \eqref{b2supGama} is satisfied. Then, there exists a positive constant $C$ such
that  every solution $(U,V)$ of \eqref{adjfirst}-\eqref{adjend}
satisfies
\begin{align}
\int_0^1\left[U^2(T,x)+V^2(T,x)\right]dx\leq
C\int_0^T\int_{\omega}U^2(t,x)dx dt. \label{observ1forc}
\end{align}
\end{theorem}
%\begin{theorem}\label{observ}
%Let $T>0$ be given. Then, there exists a positive constant $C$ such
%that  every solution $(U,V)$ of \eqref{adjfirst}-\eqref{adjend}
%satisfies
%$$
%\int_0^1\left[U^2(T,x)+V^2(T,x)\right]dx\leq
%C\int_0^T\int_{\omega}\left[U^2(t,x)+V^2(t,x)\right]dx dt.
%$$
%\end{theorem}
\begin{proof}
Multiplying the equations \eqref{adjfirst} and \eqref{21}
 respectively by $U_t$ and $V_t$ and integrating  over $(0,1)$ the sum  of the new equations we obtain
\begin{align*}
0=&\int_0^1\left[U_t^2+V_t^2\right]dx -\left[x^{\alpha_1}
U_xU_t\right]_{x=0}^1-\left[x^{\alpha_2} V_xV_t\right]_{x=0}^1
\\
&+\int_0^1 b_{11}U U_t\,dx +\int_0^1 b_{22}V V_t+\int_0^1 b_{21}V U_t\,dx\\
&+\int_0^1 b_{12}U V_t\,dx+\frac{1}{2}\frac{d}{dt}\int_0^1[x^{\alpha_1}U_x^2+x^{\alpha_2}V_x^2]\,dx.
\end{align*}
Using the Young's inequality we obtain
\begin{align*}
 \frac{1}{2}\frac{d}{dt}\int_0^1[x^{\alpha_1}U_x^2+x^{\alpha_2}V_x^2]\,dx &\leq \int_0^1(b_{11}^2+b_{12}^2)U^2dx +\int_0^1(b_{22}^2+b_{21}^2)V^2dx\\
&\leq C\int_0^1(U^2(t,x)+V^2(t,x))\,dx\\
&\leq C\int_0^1[x^{\alpha_1 -2}U^2(t,x)+x^{\alpha_2 -2}V^2(t,x)]dx.
 \end{align*}
Hence, using the Hardy-Poincar\'e inequality \eqref{Hardy} one has
\begin{align*}
\frac{d}{dt}\int_0^1[x^{\alpha_1}U_x^2+x^{\alpha_2}V_x^2]\,dx \leq C_0 \int_0^1[x^{\alpha_1}U_x^2+x^{\alpha_2}V_x^2]\,dx
\end{align*}
Hence
\begin{align*}
\frac{d}{dt}\left\{e^{-C_0t}\int_0^1[x^{\alpha_1}U_x^2+x^{\alpha_2}V_x^2]dx\right\}&\leq 0.
 \end{align*}
 Consequently, the function $t\longmapsto
e^{-C_0t}\int_0^1[x^{\alpha_1}U_x^2+x^{\alpha_2}V_x^2]dx $ is not increasing. Thus,
$$\int_0^1[x^{\alpha_1}U_x^2(T,x)+x^{\alpha_2}V_x^2(T,x)]dx\leq
e^{C_0T}\int_0^1[x^{\alpha_1}U_x^2(t,x)+x^{\alpha_2}V_x^2(t,x)]dx.
$$
Integrating over
$[\frac{T}{4},\frac{3T}{4}]$ and using the Carleman estimate \eqref{carlemanonef} one obtains
\begin{align*}
%\label{obser2}
\int_0^1[x^{\alpha_1}U_x^2(T,x)+x^{\alpha_2}V_x^2(T,x)]dx&\leq \frac{2e^{C_0T}}{T}\int_{\frac{T}{4}}^{\frac{3T}{4}}\int_0^1[x^{\alpha_1}U_x^2(t,x)+x^{\alpha_2}V_x^2(t,x)]dx dt
\nonumber
\\  &\leq C_T \int_{\frac{T}{4}}^{\frac{3T}{4}}\int_0^1s\Theta e^{2s\varphi}[x^{\alpha_1}U_x^2(t,x)+x^{\alpha_2}V_x^2(t,x)]dx dt
\nonumber
\\
&\leq C_T\int_0^T\int_\omega U^2(t,x)
% +V^2(t,x)]
dx dt,
\end{align*}
On the other hand, using hardy-Poincar\'e inequality one gets
\begin{align*}
\int_0^1[U^2(T,x)+V^2(T,x)]dx &\leq \int_0^1[x^{\alpha_1-2}U^2(T,x)+x^{\alpha_2-2}V^2(T,x)]dx
\\
& \leq C \int_0^1[x^{\alpha_1}U_x^2(T,x)+x^{\alpha_2}V_x^2(T,x)]dx
\nonumber
\end{align*}
This ends the proof.
\end{proof}
By Theorem \ref{observatos} and a classical argument one can deduce the controllability result
\begin{theorem}
%\begin{itemize}
%\item
%The linear degenerate parabolic system \eqref{eq:1}-\eqref{eq:2} with two control forces is null controllable.
%\item
If
%moreover
the assumption \eqref{b2supGama} is satisfied, then the degenerate coupled system \eqref{eq:1}-\eqref{eq:2}
%with one control force $(h_2=0)$
is null controllable.
%\end{itemize}
\end{theorem}
\section{Appendix 1}
As in \cite{bouss}, \cite{Ca-Te}, \cite{bahm}, we give the proof of the Caccioppoli's inequality for degenerate coupled systems with two different diffusion coefficients.
\begin{lemma}\label{cacci}
Let $\omega' \Subset\omega$. Then there exists a positive constant C such that
\begin{align*}
\int_0^T\int_{\omega'} \left[
U_x^2(t,x)+V_x^2(t,x)\right]e^{2s\varphi_i}dx dt  \leq
C\int_0^T\int_{\omega} \left[ U^2(t,x)+V^2(t,x)\right]e^{2s\varphi_i}dx dt.
\end{align*}
\end{lemma}
\begin{proof}
Let $\chi\in\mathcal{C}^\infty(0,1)$ such that $supp\,\chi \subset
\omega$ and $\chi \equiv 1$ on $\omega'.$ We have
\begin{align*}
0=&\int_0^T\frac{d}{dt} \left[ \int_0^1\chi^2
(U^2+V^2) e^{2s\varphi_i}dx \right] dt\\
=&
-2\int_0^T\int_0^1  \chi^2 x^{\alpha_1}U_x^2 e^{2s\varphi_i}dx dt
-2\int_0^T\int_0^1  \chi^2 x^{\alpha_2}V_x^2 e^{2s\varphi_i}dx
\\
&+ \int_0^T\int_0^1  (x^{\alpha_1}(\chi^2 e^{2s\varphi_i})_x)_x  U^2 dx +
   \int_0^T\int_0^1  (x^{\alpha_2}(\chi^2 e^{2s\varphi_i})_x)_x  V^2 dx
   \\
& -2\int_0^T\int_0^1 b_{11}\chi^2  U^2 e^{2s\varphi_i} dxdt
  -2\int_0^T\int_0^1 b_{22}\chi^2  V^2e^{2s\varphi_i}dxdt
\\
& 2\int_0^T\int_0^1 s \dot{\varphi_i}\chi^2(U^2+V^2)e^{2s\varphi_i}dxdt
 -2\int_0^T\int_0^1 (b_{12}+b_{21})\chi^2  U V e^{2s\varphi_i}dx dt
\end{align*}
Therefore, since $\chi$ is supported in $\omega$ and $\chi\equiv 1$ in $\omega'$ then, using Young inequality one obtains
\begin{align*}
\int_0^T\int_{\omega'}  (U_x^2+ V_x^2) e^{2s\varphi_i}dxdt
& \leq C \int_0^T\int_0^1  \chi^2 (x^{\alpha_1}U_x^2+x^{\alpha_2} V_x^2) e^{2s\varphi_i}dxdt
\\
&\leq C \int_0^T\int_\omega
(U^2+V^2)e^{2s\varphi_i}dxdt.
\end{align*}
 This ends the proof.
\end{proof}
\section{Appendix 2}
\begin{lemma}\label{Hardy2}
For all $\gamma <1$ and all $v$ locally absolutely continuous on $(0,1]$, continuous at $0$ and satisfying $v(0)=0$ and $\int_0^1x^\gamma v_x^2dx <+\infty$ the following Hardy-Poincar\'e inequality holds
\begin{align}\label{Hardy}
 \int_0^1x^{\gamma-2}v^2dx\leq C_\gamma \int_0^1x^\gamma v_x^2dx \quad\quad \mbox{where}\,\,\,\,
C_\gamma=\frac{4}{(1-\gamma)^2}
\end{align}
\end{lemma}
\begin{proof}
This result was proved by Cannarsa et al. in \cite{bouss} for $\gamma \in (0,1)$, but by a careful examination of the proof one can see that it remains valid for all $\gamma <1$. In fact let $\gamma <1$ and $\delta= \frac{\gamma+1}{2}$. Using Holder inequality and Fubini's theorem one has
\begin{align}
 \int_0^1x^{\gamma-2}v^2dx & =  \int_0^1x^{\gamma-2}\left(\int_0^x y^{\delta/2}v'(y) y^{-\delta/2}dy\right)^2  dx \nonumber \\
   & \leq  \int_0^1x^{\gamma-2}\left(\int_0^x y^{\delta}|v'(y)|^2dy\right)\left( \int_0^x y^{-\delta}dy\right)  dx \nonumber \\
   &= \frac{1}{1-\delta} \int_0^1 \int_0^x  x^{\gamma-\delta-1} y^{\delta}|v'(y)|^2dy  dx  \nonumber \\
   &= \frac{1}{1-\delta} \int_0^1
   \left(\int_y^1  x^{\gamma-\delta-1}dx \right)
    y^{\delta}|v'(y)|^2  dy \nonumber \\
   & \leq  \frac{1}{(1-\delta)(\delta-\gamma)} \int_0^1  y^{\gamma}|v'(y)|^2  dy  \nonumber \\
   & =  \frac{4}{(1-\gamma)^2} \int_0^1  y^{\gamma}|v'(y)|^2  dy.  \nonumber
\end{align}
This ends the proof.
\end{proof}
\section{Conclusion}
In this paper, we studied the null controllability of linear degenerate systems with two different coefficients diffusion not necessarily of the cascade form. We developed new Carleman estimates. By a standard linearization argument and fixed point, see  \cite{bahm}, \cite{bouss}, \cite{Can-Gen-2006}, \cite{Zuazua99}, one can show easily the null controllability of semilinear degenerate coupled systems with two different diffusion coefficients. In this paper we studied coupled system of two weakly degenerate equations. The cases when one of the equation is strongly degenerate systems are open.

\vspace{4cm}

E. M. Ait Ben Hassi, D\'epartement de Math\'ematiques, Facult\'e des Sciences Semlalia, LMDP, UMMISCO (IRD-UPMC), Universit\'e Cadi Ayyad, Marrakech 40000, B.P. 2390, Morocco.
\\
 E-mail: m.benhassi@ucam.ac.ma
 \\[,2cm]
 F. Ammar khodja, Laboratoire de Math\'ematiques de Besan\c{c}on, UMR CNRS 6623, Universit\'e de Franche-
Comt\'e, 25030 Besan\c{c}on Cedex, France.
\\
 E-mail:  farid.ammar-khodja@univ-fcompte.fr
 \\[,2cm]
A. Hajjaj, D\'epartement de Math\'ematiques et Informatique, Facult\'e des Sciences et Techniques, Universit\'e Hassan 1er Settat 26000, B.P. 577,  Morocco.
\\
E-mail: hajjaj$\_{}$fsts@yahoo.fr
\\[,2cm]
L. Maniar, D\'epartement de Math\'ematiques, Facult\'e des Sciences Semlalia, LMDP, UMMISCO (IRD-UPMC), Universit\'e Cadi Ayyad, Marrakech 40000, B.P. 2390, Morroco.
\\
E-mail: maniar@ucam.ac.ma

\end{document}